\documentclass[11 pt]{amsart}
\usepackage{amssymb}
\usepackage{amsmath}
\usepackage{amsthm}
\usepackage{bbold}
\usepackage{bbm}
\usepackage{dsfont}
\newtheorem{thm}{Theorem}
\newtheorem{prop}{Proposition}

\newtheorem{exam}{Example}
\newtheorem{cor}{Corollary}

\newtheorem{defi}{Definition}

\numberwithin{equation}{section}
\newcommand{\abs}[1]{\lvert#1\rvert}

\DeclareMathOperator{\unc}{\xrightarrow[]{un}}

\DeclareMathOperator{\uoc}{\xrightarrow[]{uo}}

\DeclareMathOperator{\oc}{\xrightarrow[]{o}}

\DeclareMathOperator{\nc}{\xrightarrow[]{n}}

\DeclareMathOperator{\co}{{\textrm{co}}}
\DeclareMathOperator{\cl}{{\textrm{cl}}}

\begin{document}
	
\title{AMS Journal Sample}
\author{E. Y. Emelyanov$^1$, N. Erkur{\c s}un-{\"O}zcan$^2$, S. G. Gorokhova$^3$}
\address{$^{1}$ Department of Mathematics, Middle East Technical University, 06800 Ankara, Turkey.} 
\address{$^{2}$ Department of Mathematics, Hacettepe University, 06800 Ankara, Turkey.} 
\address{$^{3}$ Sobolev Institute of Mathematics, 630090 Novosibirsk, Russia.} 
\email{eduard@metu.edu.tr, erkursun.ozcan@hacettepe.edu.tr, lanagor71@gmail.com}
\subjclass[2010]{}

\date{06.10.2017}


\title{Koml{\'o}s properties in Banach lattices}

\begin{abstract}
Several Koml{\'o}s like properties in Banach lattices are investigated. We prove that $C(K)$ fails the $oo$-pre-Koml{\'o}s property,
assuming that the compact Hausdorff space $K$ has a nonempty separable open subset $U$ without isolated points such that every 
$u\in U$ has countable neighborhood base. We prove also that for any infinite dimensional Banach lattice $E$ 
there is an unbounded convex $uo$-pre-Koml{\'o}s set $C\subseteq E_+$ which is not $uo$-Koml{\'o}s.
\end{abstract}
\maketitle

\section{Introduction}

In recent paper \cite{GTX}, the unbounded order convergence in Banach lattices was deeply studied. Among other  
things, this development has lead to study of generalizations of the Koml{\'o}s celebrated theorem \cite{Kom} 
to the Banach lattice setting. The authors of \cite{GTX} did their generalization through $AL$-representations of 
a Banach lattice with a strictly positive order continuous functional, replacing almost everywhere convergence 
by unbounded order convergence. Beside such a natural extension, many questions on generalized Koml{\'o}s properties 
are still requiring an investigation. 

In the present paper, we study Koml{\'o}s properties in more breadth settings, than for the $uo$-convergence. 
In Section 2, we define and investigate several Koml{\'o}s properties 
for different modes of boundedness and convergence in a Banach lattice. Section 3 is completely devoted to 
Koml{\'o}s properties in Banach lattices of continuous functions. Section 4 is dealing with so-called Koml{\'o}s sets. 

As the nature of the Koml{\'o}s theorem is sequential, we restrict ourselves to sequential convergences 
in Banach lattices. For unexplained terminology and notations we refer the reader to \cite{AB2003,AB,DEM,DOT,GTX,KMT}. 
In the present paper, $E$ stands for a real Banach lattice.

\section{Koml{\'o}s like properties in Banach lattices}

Let $x_n$ be a sequence in $E$ and $x\in E$. Recall that:

\begin{enumerate}
\item $x_n$ {\em converges in order} to $x$ (we write $x_n\oc x$), if there is a sequence $y_n$ decreasing to 0 
(we write $y_n\downarrow 0$) with $|x_n-x|\le y_n$ for all $n$;
\item if $x_n$ converges in norm to $x$, we write $x_n\nc x$;
\item $x_n$ is {\em unbounded order convergent} to $x$ (we write $x_n\uoc x$) if 
$\abs{x_n-x}\wedge u\oc 0$ for every $u\in E_+$;
\item if $\abs{x_n-x}\wedge u\nc 0$ for every $u\in E_+$, we write $x_n\unc x$ and say that
$x_n$ is {\em unbounded norm convergent} to $x$.
\end{enumerate}

The main motivation for the present paper is the following classical result \cite{Kom}.

\begin{thm}[Koml{\'o}s]\label{Koml}
Let $E=L_1(\mu)$, where $\mu$ is a probability measure. Then, for every norm 
bounded sequence $x_n$, there is a subsequence $x_{n_k}$ such that the Ces{\'a}ro means 
$\frac{1}{m}\sum\limits_{j=1}^{m}x_{n_{k_j}}$ of any further subsequence $x_{n_{k_j}}$ 
converge almost everywhere to some $x\in X$.
\end{thm}

The Koml{\'o}s theorem has initiated many investigations and was extended recently for Banach lattices 
\cite{GTX} by replacing a.e.-convergence with $uo$-conver\-gence. This development has inspired 
the following definition.  

\begin{defi}\label{Komprop1}
A Banach lattices $E$ is said to have $ab$-{\em Koml{\'o}s property} $($respectively, $ab$-{\em pre-Koml{\'o}s property}$)$ 
if, for every $a$-bounded sequence $x_n$ in $E$, there exist a subsequence $x_{n_k}$ and an element $x\in E$ such that 
$$
  x=b-\lim\limits_{m\to\infty}\frac{1}{m}\sum\limits_{j=1}^{m}x_{n_{k_j}}
$$ 
for any subsequence $x_{n_{k_j}}$ of $x_{n_k}$  $($respectively, the sequence $\frac{1}{m}\sum\limits_{j=1}^{m}x_{n_{k_j}}$ 
is $b$-Cauchy for any subsequence $x_{n_{k_j}}$ of $x_{n_k}$$)$.

Here, $a$-boundedness stands for $o$- or $n$-boundedness; and $b$-convergence stands for $o$-, $uo$-, $n$-, or $un$-convergence.
\end{defi}

Clearly, it suffices to check $ab$-Koml{\'o}s and $ab$-pre-Koml{\'o}s properties only for sequences in $E_+$.
Furthermore, $no$-Koml{\'o}s implies $oo$-Koml{\'o}s and, if $E$ is $\sigma$-Dedekind complete, 
they coincide with $no$-pre-Koml{\'o}s and $oo$-pre-Koml{\'o}s properties respectively. 

$nuo$-Koml{\'o}s and $nuo$-pre-Koml{\'o}s properties were introduced in \cite[Def.5.1]{GTX} under the names of 
Koml{\'o}s and pre-Koml{\'o}s properties respectively. The Koml{\'o}s Theorem \ref{Koml} has been extended further 
in \cite[Prop.5.13]{GTX} as follows.

\begin{prop}[Gao--Troitsky--Xanthos]\label{GTXProp.5.13}
Let $E$ be a regular sublattice of an order continuous Banach lattice $F$. Then $E$ has the $nuo$-pre-Koml{\'o}s property. 
Moreover, $E$ has the $nuo$-Koml{\'o}s property iff it is sequentially boundedly uo-complete.
\end{prop}

Let us mention also the next corollary (see \cite[Cor.5.14]{GTX}) of \cite[Prop.5.13]{GTX}. 

\begin{cor}[Gao--Troitsky--Xanthos]\label{GTXCor.5.14}
Every order continuous Banach lattice $E$ has the $nuo$-pre-Koml{\'o}s property. Moreover,
$E$ has the $nuo$-Koml{\'o}s property iff it is a $KB$-space.
\end{cor}

\begin{prop}\label{un-cor}
Every $o$-continuous Banach lattice $E$ has the $nun$-pre-Koml{\'o}s property. Moreover,
if $E$ is a $KB$-space then $E$ has the $nun$-Koml{\'o}s property.
\end{prop}

\begin{proof}
It follows from Corollary \ref{GTXCor.5.14} since $uo$-convergence implies $un$-conver\-gence (see \cite[Prop.2.5]{DOT}).
\end{proof}

\begin{exam}\label{Komprop2}
The Banach lattice $c$ of real convergent sequences fails the $oo$-Koml{\'o}s $($and hence $no$-Koml{\'o}s$)$ property. 
To see this, take the sequence
$x_n=\sum\limits_{j=1}^{n}e_{2j}$ in $[0,\mathbb{1}]\subset c$. Clearly, for any subsequence $x_{n_{k_j}}$, the sequence 
$\frac{1}{m}\sum\limits_{j=1}^{m}x_{n_{k_j}}$ is $uo$-divergent, and also $o$-divergent, since it is order bounded. 
It shows also that $c$ fails the $nuo$-Koml{\'o}s property. Since the sequence $\frac{1}{m}\sum\limits_{j=1}^{m}x_{n_{k_j}}$ 
is not $n$-Cauchy, $c$ fails the $nn$-pre-Koml{\'o}s property as well.

It can be easily seen that $c$ has $oo$-, $ouo$-, and $nuo$-pre-Koml{\'o}s property 
$($cf. also \cite[Cor.5.10]{GTX}$)$.
\end{exam}

\begin{prop}\label{Komprop3}
Any $o$-continuous Banach lattice $E$ has $oo$- and $on$-Koml{\'o}s property.
\end{prop}

\begin{proof} 
Let $x_n\in [-u,u]$ for all $n$. By \cite[Cor.5.14]{GTX}, $E$ has 
$nuo$-pre-Koml{\'o}s property, and hence $ouo$-pre-Koml{\'o}s property. So, there 
exists a subsequence $x_{n_k}$ such that any sequence $\frac{1}{m}\sum\limits_{j=1}^{m}x_{n_{k_j}}\subset[-u,u]$ 
is $uo$-Cauchy, and hence $o$-Cauchy. By $o$-continuity of the norm, $E$ is Dedekind complete. It follows 
that there exist $y\in E$ with $y=o-\lim\limits_{m\to\infty}\frac{1}{m}\sum\limits_{j=1}^{m}x_{n_{k_j}}$
for any subsequence $x_{n_{k_j}}$ of $x_{n_k}$. Using $o$-continuity once more, we get
$y=n-\lim\limits_{m\to\infty}\frac{1}{m}\sum\limits_{j=1}^{m}x_{n_{k_j}}$.
\end{proof}

\section{Koml{\'o}s properties in Banach lattices of continuous functions}

Notice that, in $E=C(K)$: $oo$-Koml{\'o}s, $no$-Koml{\'o}s, $ouo$-Koml{\'o}s, and $nuo$-Koml{\'o}s properties coincide. 
The same is true for $oo$-, $no$-, $ouo$-, and $nuo$-pre-Koml{\'o}s properties. Furthermore, $nn$-Koml{\'o}s 
property coincides with both $nn$-pre-Koml{\'o}s and $on$-pre-Koml{\'o}s properties.

In view of Example \ref{Komprop2}, the Banach lattice $c\cong C(\mathbb{N}\cup\{\infty\})$ fails the $oo$-Koml{\'o}s property, but still 
has the $oo$-pre-Koml{\'o}s property. We point out that the one-point compactification $\mathbb{N}\cup\{\infty\}$ of $\mathbb{N}$ is a separable 
compact metric space in which all points except $\infty$ are isolated. 

It was mentioned in \cite[Ex.5.3]{GTX} that it was still unknown whether or not $C[0,1]$ has the $nuo$-pre-Koml{\'o}s property. 
Here, we clarify the situation with Banach lattices $C(K)$ for a large class of compact Hausdorff spaces.

\begin{thm}\label{Komprop4}
Let $K$ be a compact Hausdorff space without isolated points in which there exist 
two distinct sequences $t_n$ and $t'_n$ such that $\cl\{t_n\}_{n=1}^{\infty}=\cl\{t'_n\}_{n=1}^{\infty}=K$.
Then $C(K)$ fails the $oo$-pre-Koml{\'o}s property.
\end{thm}

\begin{proof}
Define $f_k(t)$ on $\{t_n\}_{n=1}^{k}\cup\{t'_n\}_{n=1}^{k}$ to be equal to 1 if $t=t_1,...,t_k$ and
$f_k(t)=0$ if $t=t'_1,...,t'_k$. Then extend each $f_k(t)$ continuously to whole $K$ so that 
$f_k(K)\subseteq [0,1]$. It is easy to see that, for any subsequence $f_{k_j}$ the sequence 
$$
  g_m=\frac{1}{m}\sum\limits_{j=1}^{m}f_{k_j} \ \ \ (m\in\mathbb{N})
$$
of Ces{\'a}ro means is not $oo$-Cauchy. 
\end{proof}

\begin{thm}\label{Komprop5}
Let $K$ be a compact Hausdorff space with a nonempty separable open subset $U\subset K$ without isolated points
such that every $u\in U$ has countable neighborhood base. Then $C(K)$ fails the $oo$-pre-Koml{\'o}s property.
\end{thm}

\begin{proof}
Let $D=\{d_k\}_{k=1}^{\infty}$ be a countable dense subset of a nonempty open subset $U\subset K$ without isolated points.
Without lost of generality, we may choose countable neighborhood bases $B_d=\{U_d^n\}_{n=1}^{\infty}$ of elements $d\in D$ such that 
$$
  U_d^{n+1}\subseteq U_d^n \ \ \ \ (\forall d\in D, n\in\mathbb{N})
$$ 
and $d_m\not\in U_{d_k}^k$ for $m<k$.

We choose a sequence $\{d_{n_k}\}_{k=1}^{\infty}$ of distinct elements of $D$ as follows
$$
  d_{n_1}\in U_{d_1}^1, d_{n_2}\in U_{d_1}^2, d_{n_3}\in U_{d_2}^2, d_{n_4}\in U_{d_1}^3, d_{n_5}\in U_{d_2}^3, d_{n_6}\in U_{d_3}^3, ... \ .
$$ 
It is an easy exercise to show that $\cl\{d_{n_{2k-1}}\}_{k=1}^{\infty}=\cl\{d_{n_{2k}}\}_{k=1}^{\infty}=\cl U$.

By Theorem \ref{Komprop4}, there is a sequence $f_k\in\cl U$ such that $f_k(\cl U)\subseteq [0,1]$ for all $k$ with the property that 
for any subsequence $f_{k_j}$ the sequence 
$$
  g_m=\frac{1}{m}\sum\limits_{j=1}^{m}f_{k_j} \ \ \ (m\in\mathbb{N})
$$
of Ces{\'a}ro means is not $oo$-Cauchy in $C(\cl U)$. Now, extend $f_k$ to $\bar{f}_k\in C(K)$ so that $\bar{f}_k(K)\subseteq [0,1]$ for all $k$.
Clearly, for any subsequence $\bar{f}_{k_j}$ the sequence 
$$
  y_m=\frac{1}{m}\sum\limits_{j=1}^{m}\bar{f}_{k_j} \ \ \ (m\in\mathbb{N})
$$
is not $oo$-Cauchy in $C(K)$.
\end{proof}

\begin{cor}\label{Komprop6}
For a compact metric space $K$ possessing a nonempty separable open subset without isolated points, 
the Banach lattice $C(K)$ fails the $oo$-pre-Koml{\'o}s property.
\end{cor}

Note that $\ell_{\infty}\cong C(\beta{\mathbb{N}})$ has $oo$-Koml{\'o}s property (cf. e.g., \cite[Ex.5.11]{GTX})
From the other hand side, \cite[Ex.5.2]{GTX} implies that $\ell_{\infty}(\Gamma)$ fails the $oo$-pre-Koml{\'o}s property, 
whenever $\mathrm{Card}(\Gamma)\ge c$.

\section{Koml{\'o}s sets}

The converse of the Koml{\'o}s Theorem \ref{Koml} has been proved in \cite[Thm.2.1]{Len}, namely
Lennard has proved that: {\em for a probability measure $\mu$, every convex $C\subset L_1(\mu)$ must be norm bounded
provided that $C$ satisfies the property: for every sequence $x_n$ in $C$ there exist 
a subsequence $x_{n_k}$ and an $x\in E$, with 
$$
  x=uo-\lim\limits_{m\to\infty}\frac{1}{m}\sum\limits_{j=1}^{m}x_{n_{k_j}}
$$
for any subsequence $x_{n_{k_j}}$ of $x_{n_k}$}. Subsets of any Banach lattice $E$ satisfying
above property are called {\em Koml{\'o}s sets} in \cite[Def.5.22]{GTX}. This motivates the 
following definition.

\begin{defi}\label{Komset1}
We say that $C\subset E$ is an $o$-, $uo$-, $n$, or $un$-{\em Koml{\'o}s set} $($respecti\-ve\-ly, $o$-, $uo$-, $n$-, 
or $un$-pre-{\em Koml{\'o}s set}$)$ if, for every sequence $x_n$ in $C$, there exist a subsequence $x_{n_k}$ and $x\in E$ such that, 
for any further subsequence $x_{n_{k_j}}$, the sequence $g_m=\frac{1}{m}\sum\limits_{j=1}^{m}x_{n_{k_j}}$ is $o$-, $uo$-, $n$-, 
or $un$-convergent to $x$ $($respectively, $g_m$ is $o$-, $uo$-, $n$-, or $un$-Cauchy$)$.
\end{defi}

The main result of paper \cite{GTX} concerning Koml{\'o}s sets \cite[Thm.5.23]{GTX} can be read as follows.

\begin{thm}[Gao--Troitsky--Xanthos]\label{Komset2}
Let $E$ be a Banach lattice with the projection property. If $E_{oc}\tilde{}$ is a norming subspace of $E^*$, 
then any convex $uo$-Koml{\'o}s set $C$ in $E$ is norm bounded.
\end{thm}

Below, in Proposition \ref{Komset4} we show that in arbitrary Banach lattice $E$ every convex $uo$-Koml{\'o}s set $C\subseteq E_+$ is norm bounded.

The following result shows that in \cite[Thm.2.1]{Len} and in \cite[Thm.5.23]{GTX} $uo$-Koml{\'o}s sets can not be replaced by $uo$-pre-Koml{\'o}s set. 

\begin{thm}\label{Komset3}
In any infinite dimensional Banach lattice $E$, there is an unbounded convex $uo$-pre-Koml{\'o}s set $C\subseteq E_+$.
\end{thm} 

\begin{proof}
By \cite[Cor.3.6, Cor.3.13]{GTX}, any disjoint sequence is a $uo$-Koml{\'o}s set.
Take a disjoint sequence $d_n$ in $E_+$ such that $\|d_n\|=n$ for all $n$.
Let $x_i=\sum\limits_{n=1}^{\infty}\alpha_n^i d_n$ be a sequence in the convex hull $C=\co\{d_n\}_{n=1}^{\infty}$. 
By diagonal argument, it is easy to find a subsequence $x_{i_j}$ satisfying 
$$
  \lim\limits_{n\to\infty}\alpha_n^{i_j}=\beta_n \ \ \ \ (\forall n).
	\eqno(1)
$$ 
By choosing further subsequence, if necessary, we may suppose that  
$$
  |\alpha_n^{i_{j}}-\beta_n|\le 1 \ \ \ \ (\forall j\ge n).
	\eqno(2)
$$
For $u=o-\sum\limits_{n=1}^{\infty} nd_n$, $y=o-\sum\limits_{n=1}^{\infty}\beta_n d_n$ in the universal completion $E^u$ (cf. \cite[Def.7.20]{AB2003}) of $E$,
by using (1) and (2) we get that 
$$
  |x_{i_j}-y|\le \frac{1}{n}u
	\eqno(3)
$$
for big enough $j$. In view of (3), the sequence $x_{i_j}$ $o$-converges to $y=o-\sum\limits_{n=1}^{\infty}\beta_n d_n$ in $E^u$, and hence 
is $uo$-Cauchy in $E$ by \cite[3.12]{GTX}. 
It follows that every further subsequence $x_{i_{j_l}}$ is $uo$-Cauchy, and hence, by \cite[3.13]{GTX}, the sequence 
$g_m=\frac{1}{m}\sum\limits_{l=1}^{m}x_{i_{j_l}}$ of its Ces{\'a}ro means is $uo$-Cauchy as well. 

Thus, $C$ is a norm unbounded convex $uo$-pre-Koml{\'o}s set.  
\end{proof}

\begin{prop}\label{Komset4}
In any Banach lattice $E$ every convex $uo$-Koml{\'o}s set $C\subseteq E_+$ is norm bounded.
\end{prop}

\begin{proof}
Let $C\subseteq E_+$ be a norm unbounded convex set. We are going to show that $C$ is not $uo$-Koml{\'o}s.

Choose a sequence $x_n\in C$ so that $\|x_1\|\ge 4$ and 
$$
  \frac{1}{2^n}\|x_n\|\ge 2\sum_{k=1}^{n-1}\frac{1}{2^k}\|x_k\| \ \ \ \ (\forall n>1).
$$
Define an increasing sequence $z_n\in E_+$ as follows: 
$$
  z_n=\sum_{k=1}^{n}\frac{1}{2^k}x_k. 
$$  
Clearly, every subsequence $z_{n_k}$ of $z_n$ is not $uo$-convergent, since otherwise $z_{n_k}\uoc z\in E_+$ and $z-z_{n_k}=|z-z_{n_k}|\wedge z\oc 0$, 
which means that $z_{n_k}\uparrow z$, and hence $\|z_{n_k}\|\le\|z\|<\infty$, violating
$$
  \|z_n\|\ge\frac{1}{2^n}\|x_n\|\ge 2\sum_{k=1}^{n-1}\frac{1}{2^k}\|x_k\|\ge 2\|z_{n-1}\|\ge 2^n\ \ \ \ (\forall n\ge 1).
$$ 
The sequence $w_m=\frac{1}{m}\sum\limits_{j=1}^{m}z_{n_k}$ is also increasing and $\|w_m\|\to\infty$. 
The similar argument shows that $w_m$ is not $uo$-convergent. Let
$$
  y_n=\frac{1}{2^n}x_1+z_n\in C\ \ \ \ (\forall n\ge 1).
$$
Since $\frac{1}{m}\sum\limits_{k=1}^{m}\frac{1}{2^{n_k}}x_1\le\frac{1}{m}x_1\oc 0$, the sequence of the Ces{\'a}ro means
$$
  \frac{1}{m}\sum\limits_{k=1}^{m}y_{n_k}=\frac{1}{m}\sum\limits_{k=1}^{m}\frac{1}{2^{n_k}}x_1+\frac{1}{m}\sum\limits_{k=1}^{m}z_{n_k}
	=w_m+\frac{1}{m}\sum\limits_{k=1}^{m}\frac{1}{2^{n_k}}x_1 \ \ \ (m\in\mathbb{N})
$$
is not $uo$-convergent. It shows that $C$ is not $uo$-Koml{\'o}s.
\end{proof}

Note that the similar argument, as in the proof of Proposition \ref{Komset4}, shows that in any Banach lattice 
$E$ every convex $un$-Koml{\'o}s set $C\subseteq E_+$ is norm bounded.

The next result follows directly from Theorem \ref{Komset3} and Proposition \ref{Komset4}. 

\begin{thm}\label{Komset6}
Let $E$ be an infinite dimensional Banach lattice. Then there exists a norm unbounded convex $uo$-pre-Koml{\'o}s set $C\in E_+$ which is not $uo$-Koml{\'o}s.
\end{thm}

\begin{cor}\label{lastcor}
Let $E$ be an Banach lattice. The following conditions are equivalent$:$

$(1)$ $\dim(E)<\infty$$;$

$(2)$ $E$ is $uo$-complete$;$

$(3)$ $E$ is sequentially $uo$-complete$;$

$(4)$ every $uo$-pre-Koml{\'o}s set is $uo$-Koml{\'o}s$;$

$(5)$ every convex $uo$-pre-Koml{\'o}s set $C\in E_+$ is norm bounded.

\end{cor}

\begin{proof}
Implications $(1)\Rightarrow(2)\Rightarrow(3)$ are trivial, and $(3)\Rightarrow(4)$ easily follows from Definition \ref{Komset1}.

$(4)\Rightarrow(5)$: Let $C\in E_+$ be a convex $uo$-pre-Koml{\'o}s set. Then $C$ is $uo$-Koml{\'o}s by the assumption.
Proposition \ref{Komset4} ensures that $C$ is norm bounded.

$(5)\Rightarrow(1)$: It is exactly Theorem \ref{Komset3}.
\end{proof}

\end{document}